\documentclass[12pt]{amsproc}

\usepackage{amsmath, amsthm, amssymb, amsfonts, mathrsfs, amscd, tikz}
\usepackage{enumerate}
\usetikzlibrary{matrix,arrows}

\def\AA{{\mathbb A}}

\def\QQ{{\mathbb Q}}
\def\PP{{\mathbb P}}
\def\QQ{{\mathbb Q}}

\def\ZZ{{\mathbb Z}}

\def\0{{\mathbf 0}}
\def\1{{\mathbf 1}}

\def\Ical{{\mathcal I}}

\def\Mcal{{\mathcal M}}

\def\Ocal{{\mathcal O}}
\def\Pcal{{\mathcal P}}

\def\Vcal{{\mathcal V}}
\def\Wcal{{\mathcal W}}

\def\mfrak{{\mathfrak m}}

\def\Kbar{{\bar K}}

\def\Gal{\mathrm{Gal}}

\def\PGL{\mathrm{PGL}}
\def\GL{\mathrm{GL}}

\def\PrePer{\mathrm{PrePer}}
\def\Per{\mathrm{Per}}

\def\Fix{\mathrm{Fix}}

\def\Res{\mathrm{Res}}
\def\Hom{\mathrm{Hom}^n_d}

\def\uf{\mathrm{uf}}
\def\uc{\mathrm{uc}}

\def\Ratfd2{\mathrm{Rat}_{d,2}^{\uf}}
\def\Mfd2{\mathcal{M}_{d,2}^{\uf}}
\def\Rfd2{\mathcal{R}_{d,2}^{\uf}}
\def\Rcd2{\mathcal{R}_{d,2}^{\uc}}

\def\Frac{\mathrm{Frac}}
\def\Tw{\mathrm{Twist}}

\theoremstyle{plain}

\newtheorem{thm}{Theorem}

\newtheorem{prop}[thm]{Proposition}
\newtheorem{lem}[thm]{Lemma}

\theoremstyle{definition}
\newtheorem*{dfn}{Definition}

\newtheorem{rem}{Remark}

\begin{document}

\title[Twists of rational morphisms]{A dynamical Shafarevich theorem for twists of rational morphisms}

\author{Brian Justin Stout}

\address{Brian Justin Stout; Department of Mathemaics; U.S. Naval Academy; Chauvenet Hall; Annapolis, MD. 21401-1363 U.S.A.}

\email{bstout@gc.cuny.edu}

\thanks{{\em Date of last revision:} 22 August 2013}

\begin{abstract}
Let $K$ denote a number field, $S$ a finite set of places of $K$, and $\phi:\PP^n\rightarrow\PP^n$ a rational morphism defined over $K$. The main result of this paper proves that there are only finitely many twists of $\phi$ defined over $K$ which have good reduction at all places outside $S$.  This answers a question of Silverman in the affirmative.
\end{abstract}

\maketitle


\section{Introduction}\label{Introduction}

Let $K$ be a number field and $S$ a finite set of places of $K$ which includes all the Archimedian places.  For arithmetic objects defined over $K$ one can pose questions about the number of $K$-isomorphism classes which have good reduction at all places not in $S$.  Shafarevich initially asked this question for elliptic curves over $K$ and proved the number of classes to be finite (see \cite{silverman:aec}). Faltings subsequently proved the same for abelian varieties (see \cite{MR718935}).

The similarity between the arithmetic theory of elliptic curves and the arithmetic theory of rational morphisms has prompted many questions about dynamical analogues of Shafarevich's theorem for rational morphisms on projective space.  A similar finiteness result for rational morphisms can easily be seen to be false: any monic polynomial defined over $\Ocal_K$ on $\PP^1$ exhibits everywhere good reduction. For each $d\geq 2$ it is easy to show that there are infinitely many such polynomials which are non-isomorphic.

The notion of a dynamical Shafarevich theorem was studied first by Szpiro and Tucker in \cite{MR2435841} for rational maps on $\PP^1(K)$.  Szpiro-Tucker weaken the notion of $K$-isomorphism of rational maps by allowing pre-composition and post-composition by different elements of $\PGL_2$ and altering the notion of good reduction.  They subsequently obtain a finiteness result for rational maps with \textit{critical good reduction}.  In \cite{PetscheCSRM} Petsche obtains a different finiteness theorem by restricting the families of rational maps on $\PP^1$ under consideration, but retains the normal notion of $K$-isomorphism for rational maps. In a previous paper the author and Petsche consider whether set of quadratic rational maps of $\PP^1$ with good reduction outside $S$ is Zariski dense in the moduli space $\Mcal_2$ of quadratic rational maps (see \cite{PetscheStout1}). On the contrary, quadratic rational maps with everywhere good reduction over $\ZZ$ are Zariski dense in $\Mcal_2(\QQ)$.  By restricting the class of rational maps to those with two unramified fixed points and strengthening the notion of good reduction, the author and Petsche prove a Zariski non-density result.

In the present paper we consider a Shafarevich question originally posed by Silverman in Chapter Three of \cite{silverman:msad} regarding the finiteness of rational morphisms $\psi:\PP^n\rightarrow\PP^n$ defined over $K$ of degree $d\geq 2$ which have good reduction at all places $v\not\in S$ and are twists of a given rational morphism $\phi$ defined over $K$. 

We say that two rational morphisms $\phi ,\psi:\PP^n\rightarrow\PP^n$ defined over $K$ are $\Kbar$-\textit{isomorphic} if $\psi=\phi^f$ for $f\in\PGL_{n+1}(\Kbar)$ and $K$-\textit{isomorphic} if $\psi=\phi^f$ for $f\in\PGL_{n+1}(K)$. Here the notion $\phi^f$ denotes the conjugation of $\phi$ by $f$ (see Section \ref{Preliminaries}). These notions are clearly equivalence relations and we denote the set of rational morphisms which are $\Kbar$-isomorphic to $\phi$ by
\begin{equation*}
[\phi]=\{\psi\textrm{ defined over } K|\psi=\phi^f\textrm{ for some }f\in\PGL_{n+1}(\Kbar)\}
\end{equation*}
and the set of rational morphisms which are $K$-isomorphic to $\phi$ by
\begin{equation*}
[\phi]_K=\{\psi\textrm{ defined over } K|\psi=\phi^f\textrm{ for some }f\in\PGL_{n+1}(K)\}
\end{equation*}
We then define the set of twists of $\phi$ as the set of $K$-isomorphism classes of rational morphisms $\psi$ defined over $K$ which are $\Kbar$-isomorphic to $\phi$
\begin{equation*}
\Tw(\phi/K)=\lbrace [\psi]_K |\psi\text{ is defined over } K\text{ and } [\phi]=[\psi]\rbrace
\end{equation*}

We say that a morphism $\phi:\PP^n\rightarrow\PP^n$ defined over $K$ has good reduction at a non-Archimedean place $v$ of $K$ if there exists some conjugate $\psi=\phi^f$ defined over $K$ for $f\in\PGL_{n+1}(\Kbar)$ such that $\psi$ extends to a morphism of the same degree over the ring of $v$-adic integers $\Ocal_v$. For an equivalent and more precise definition, see Section \ref{Preliminaries}.  We remark that the notion of good reduction of a morphism is $K$-isomorphism invariant, and therefore the notion of good reduction of a twist at a place $v$ is immediate.

The principle theorem of this paper is the following.

\begin{thm}\label{MainThm}
Let $\phi:\PP^n\rightarrow\PP^n$ be a rational morphism of degree $d\geq 2$ defined over $K$ and let $S$ be a finite set of places including the Archimedean places.  Let 
\begin{center}
$\Vcal(S)=\lbrace [\psi]_K\in\Tw(\phi/K)|[\psi]_K\text{ has good reduction outside S}\rbrace$
\end{center}
Then $\Vcal(S)$ is finite.
\end{thm}

The main theorem is proved by contradiction.  Assuming that $\Vcal(S)$ is infinite will give an infinite sequence of automorphisms 
\begin{equation*}
f_i\in\PGL_{n+1}(\Kbar)
\end{equation*}
which produce infinitely many distinct twists $[\psi_i]_K$ defined over $K$, each with good reduction at all $v\not\in S$. Fix an integer $M>n+1$ and let $\PrePer(\phi,M)$ denote the set of pre-periodic points for $\phi$ with forward orbit of size less than or equal to $M$. Then for each $i=1,2,\ldots$ the automorphism $f_i$ defines a bijection between $\PrePer(\phi,M)$ and $\PrePer(\psi_i,M)$ by $P\mapsto f(P)$. The premise of the proof is that, after passing to a suitable infinite subsequence and letting $M$ become sufficiently large, the the sets $\PrePer(\psi_i,M)$ can be assumed to all be equal from the Diophantine finiteness theorem which we will prove in Section 3. The main ingredient of this step will follow from a finiteness theorem for forms with unit discriminant due to Evertse-Gy\H{o}ry in \cite{EvertseGyory}. It will follow that each $f_i$ is a bijection on a finite subset of $\PP^n$ and, after concluding that some subset of $\PrePer(\phi,M)$ is in general position, only finitely many $f_i$ can exist.

\emph{Acknowledgements} The author would like to thank Clayton Pestche for his guidance over the last several years and for his generous friendship.  He would also like to thank Lucien Szpiro, Victor Kolyvagin, and Ken Kramer for their support under NSF grant DMS-0739346.


\section{Preliminaries}\label{Preliminaries}
\subsection{Review of rational morphisms on projective space.} We will fix the notation and definitions regarding the dynamics of rational morphisms on projective space; for more details see \cite{silverman:ads}.

Fix coordinates $(X_0:\ldots :X_n)$ of $\PP^n(\Kbar)$. An arbitrary rational morphism $\phi:\PP^n\rightarrow\PP^n$ defined over $\bar{K}$ is given by an $n+1$ tuple
\begin{equation*}
\phi(X_0:\ldots :X_n)=(F_0(X_0,\ldots ,X_n),\ldots ,F_n(X_0,\ldots ,X_n))
\end{equation*}
 where $F_i(X_0,\ldots ,X_n)$ is a homogeneous polynomial of degree $d$ for $i=0,\ldots, n$ and $F_0,\ldots ,F_n$ have no nontrivial common solutions. 
 
Using a multi-index $j=(j_0,\ldots ,j_n)$ where each $0\leq j_k\leq d$ and $j_0+\cdots +j_n=d$ we may write
\begin{equation*}
F_i=\underset{j}{\sum} a_{ij}X^j
\end{equation*}
with coefficients $a_{ij}\in\Kbar$.  Here  $X^j$ denotes the monomial $X_0^{j_0}\cdots X_n^{j_n}$. There are $\binom{n+d}{d}$ monomials of degree $d$ in $n+1$ variables and so $\phi$ can be identified with a point $(a_{0j}:\ldots :a_{nj})\in\PP^N$ where $N=N(n,d)=(n+1)\binom{n+d}{d}-1$. Conversely, any point of $\PP^N$ determines a rational map $\phi:\PP^n\rightarrow\PP^n$, although this map may not be a morphism.  The requirement that $\phi$ be a morphism is equivalent to the non-vanishing of the resultant polynomial, i.e. that $\Res(F_0,\ldots ,F_n)\neq 0$, where the resultant polynomial is a multi-homogeneous polynomial over $\ZZ$ in the coefficients $a_{ij}$.  

There is a natural $\PGL_{n+1}(\bar{K})$ action on rational morphisms which sends $f\in\PGL_{n+1}(\bar{K})$ and $\phi$ to $\phi^f=f\circ\phi\circ f^{-1}$. When $\phi$ is a rational morphism it may be iterated and we write $\phi^n=\phi\circ\phi\circ\cdots\circ\phi$ to denote the $n^\mathrm{th}$ iterate of $\phi$.  The action of conjugation is compatible with iteration in the sense that $(\phi^f)^n=(\phi^n)^f$.

There are natural sets of points in $\PP^n$ which can be associated to a rational morphism $\phi:\PP^n\rightarrow\PP^n$.  A point $P\in\PP^n(\Kbar)$ is \emph{periodic} if $\phi^m(P)=P$ for some positive integer $m\geq 1$ and \emph{pre-periodic} if some iterate $\phi^n(P)$ is periodic.  Equivalently, $P$ is pre-periodic if its forward orbit is finite.  

We use the notations $\Per(\phi),\PrePer(\phi)$ to denote the set of all periodic points or pre-periodic points, respectively, for a fixed rational morphism $\phi$.  We also use the notation 
\begin{equation}
\PrePer(\phi,M)=\{ P\in\PP^n |\text{   } |\Ocal_\phi(P)|\leq M\}
\end{equation}
to denote the set of pre-periodic points with forward orbit of length at most $M$.  Here $\Ocal_\phi(P)$ denotes the forward orbit of $P$ under $\phi$.  We use $\Fix(\phi)=\{P\in\PP^n|\phi(P)=P\}$ to denote the fixed points of $\phi$.

\subsection{Review of twists of rational morphisms.} Suppose $\phi,\psi:\PP^n\rightarrow\PP^n$ are rational morphisms of degree $d$ defined over $K$.  We say that $\phi$ and $\psi$ are $\Kbar$-\textit{isomorphic} if $\psi=\phi^f$ for some $f\in\PGL_{n+1}(\Kbar)$ and that they are $K$-\textit{isomorphic} if $\psi=\phi^f$ for some $f\in\PGL_{n+1}(K)$.  We denote the sets of $\Kbar$-isomorphic and $K$-isomorphic rational morphisms by

\begin{equation}
\begin{aligned}
\left[ \phi\right]  &=\lbrace \phi^f | f\in \PGL_2(\bar{K})\rbrace\\
[\phi]_K&=\lbrace \phi^f | f\in \PGL_2(K)\rbrace
\end{aligned}
\end{equation}

\begin{dfn}
Let $\phi$ and $\psi$ be two rational morphisms of degree $d$ over $K$. We say that the $K$-isomorphism classes $[\phi]_K$ and $[\psi]_K$ are twists if $[\phi]=[\psi]$.  The twist $[\psi]_K$ is a \emph{non-trivial} twist if $[\phi]_K\neq[\psi]_K$. We may also abuse this definition and call $\psi$ a twist of $\phi$ if they are $\Kbar$-isomorphic but not $K$-isomorphic.
\end{dfn}

Two rational morphisms $\phi,\psi$ defined over $K$ which are twists have identical geometric properties as morphisms on $\PP^n(\Kbar)$ but may have significantly different arithmetic properties as morphisms on $\PP^n(K)$. Let $\Hom$ denote the parameter space of rational morphisms of degree $d$ on $\PP^n$. Then if $\phi,\psi$ are twists they descend to the same point in the moduli space $\Mcal^n_d$ under the quotient map
\begin{equation*}
\Hom\rightarrow\Hom/\PGL_{n+1}=\Mcal^n_d.
\end{equation*}
When $n=1$, it is essential to note that the theorem proved in this paper is fundamentally different than the other dynamical-Shafarevich results of \cite{PetscheCSRM},\cite{PetscheStout1},and \cite{MR2435841} in the sense that the finiteness theorem of this paper holds only within a single $\Kbar$-isomorphism class of morphisms defined over $K$.

\subsection{Review of number theoretic preliminaries.}

Let $M_K$ denote the places of the number field $K$.  For any place $v\in M_K$ let $|\cdot|_v$ denote any absolute value on $K$ associated to $v$. If $v$ is non-Archimedean, let $K_v$ denote the completion of $K$ with respect to $v$ and
\begin{equation*}
\begin{aligned}
\Ocal_v &=\lbrace x\in K_v | |x|_v\leq 1\rbrace\\
\Ocal^\times_v &=\lbrace x\in K_v | |x|_v= 1\rbrace
\end{aligned}
\end{equation*}
denote the subring of $v$-integral elements and the group of $v$-adic units in $\Ocal_v$, respectively. $\Ocal_v$ is a discrete local ring with maximal ideal $\mfrak_v=\lbrace x\in K | |x|_v<1\rbrace$.  Let $\Ocal_v\rightarrow k_v=\Ocal_v/\mfrak_v$ be the reduction map on to the residue field $k_v$.  For $x\in \Ocal_v$ we denote the image of this map by $\tilde{x}_v$ or just $\tilde{x}$ if $v$ is understood. 

For a rational morphism $\phi:\PP^n\rightarrow\PP^n$ defined over $K$ and $v\in M_K$ we can define the \emph{reduction} of $\phi$ at the place $v$ in the following manner. Following the natural embedding $K\rightarrow K_v$ one can consider $\phi$ to be a rational morphism over $K_v$.  As $K_v=\Frac(\Ocal_v)$ and $\Ocal_v$ is a PID, one can choose homogeneous coefficients for $\phi=(a_{0j}:\ldots :a_{nj})$ as a point in $\PP^N$ such that $|a_{ij}|_v\leq 1$ and $\max_{i,j}(|a_{ij}|_v)=1$.

\begin{dfn}  The \emph{reduction of $\phi$ at $v$} is the rational map \begin{equation*}
\tilde{\phi}_v=(\tilde{a_{0j}}:\ldots :\tilde{a_{nj}})\in\PP^N(k_v)
\end{equation*}
\end{dfn}

This reduction is independent of the choice of homogeneous coordinates. The reduction of a morphism may or may not be a morphism over the residue field. 

\begin{dfn}
For a rational morphism $\phi:\PP^n\rightarrow\PP^n$ of degree $d$ we say that $\phi$ has \emph{good reduction} at $v$ if there exists some $f\in\PGL_{n+1}(K)$  such that $\widetilde{\phi^f}_v:\PP^n\rightarrow\PP^n$ is a morphism defined over $k_v$ and $\deg(\widetilde{\phi^f}_v)=d$. We say $\phi$ has \emph{bad reduction} otherwise.
\end{dfn}

By definition, the notion of good reduction is seen to be a $\PGL_2(K)$-invariant concept and is therefore well defined for a $K$-equivalence class $[\phi]_K$.

In this paper $S$ denotes a finite subset of $M_K$ which includes all of the Archimedean places, $\Ocal_S$ the $S$-integers of $K$, $\Ocal^\times_S$ the $S$-unit group of $K$.  More specifically,
\begin{equation*}
\begin{aligned}
\Ocal_S=&\lbrace x\in K | \vert x \vert_v\leq 1\textrm{ for all }v\not\in S\rbrace\\
\Ocal^\times_S=&\lbrace x\in K | \vert x \vert_v= 1\textrm{ for all }v\not\in S\rbrace
\end{aligned}
\end{equation*}

\begin{dfn}
The \emph{absolute $S$-integers of $\Kbar$} will consist of all elements of $\Kbar$ which are $w$-integral for every place $w$ of $\Kbar$ whose restriction to $K$ is not in $S$.
The absolute $S$-integers of $\Kbar$ are denoted by $\overline{\Ocal}_S$.
\end{dfn}

For a point $P=(p_0:\ldots :p_n)\in\PP^n(K)$ and a place $v$ we say that the coordinates are \emph{normalized} with respect to $v$, or $v$-normalized, if $|p_i|_v\leq 1$ for $0\leq i\leq n$ and $|p_i|=1$ for some $i$.  The following lemma is well known, so we omit the proof.

\begin{lem}
Let $P\in\PP^n(K)$ and $S$ be sufficiently large such that $\Ocal_S$ is a principal ideal domain.  Then there exists coordinates $(p_0:\ldots :p_n)$ for $P$ which are $v$-normalized for all $v\not\in S$.
\end{lem}

We will also call such coordinates normalized, and context will make it clear whether we refer to a single place $v$ or to all places $v\not\in S$. 


\section{A Diophantine result.}\label{FiniteOrbitSection}

Fix a number field $K$, an algebraic closure $\bar{K}$, and a morphism $\phi:\PP^n\rightarrow\PP^n$ of degree at least $d\geq 2$ defined over $K$. Fix projective coordinates $(X_0:\cdots:X_n)$ on $\PP^n$ and let $v$ denote a non-Archimedean place of $K$.

\begin{dfn}\label{DecompForm}
A form $F\in K[X_0,\cdots ,X_n]$ is called \emph{decomposable} if it can be factored over its splitting field as $F=\lambda\ell^{k_1}_1\cdots\ell^{k_t}_t$ for $\lambda\in K^*$, $\ell_1,\cdots ,\ell_t$ are pair-wise non-proportional homogeneous linear polynomials over $\bar{K}$ and $k_1,\cdots ,k_t$ are positive integers such that $k_1+\cdots +k_t=\deg(F)$.
\end{dfn}

This type of form is studied by Evertse and Gy\H{o}ry in \cite{EvertseGyory}. Each decomposable form has an associated discriminant, which is a fractional ideal of $\Ocal_S$. From this point on, by ``ideal" we mean fractional ideal.

Let $\Vcal\subset\PP^n(\Kbar)$ be a finite subset $|\Vcal |> n$. Let $P_0,\ldots ,P_n\in\Vcal$ be a collection of $n+1$ points. We define $\det(P_0,\ldots ,P_n)$ to be the determinant of the $n+1$ by $n+1$ matrix of the coordinates of the $P_i$.  This determinant depends on the choice of representation used to compute it. If the coordinates of $P_i$ are replaced by $(rp_{i0}:\ldots :rp_{in})$ with $r\in K$, then the determinant changes by a factor of $r$ as well.  Any linear homogeneous form $\ell\in K[X_0,\ldots ,X_n]$ in $n+1$ variables can be identified as a point in projective space by 

\begin{equation}
\ell=\Sigma_i p_iX_i\mapsto P=(p_0:\cdots :p_n).
\end{equation}

We can therefore define $\det(\ell_0,\ldots ,\ell_n)$ for $n+1$ linear forms of $K[X_0,\ldots ,X_n]$ in the analogous way.

\begin{dfn}
Suppose $F(X_0,...,X_n)$ is a decomposable form and $F=\lambda\ell^{k_1}_1\cdots\ell^{k_t}_t$ in the splitting field $L$ of $F$.  Then the \emph{discriminant} of $F$, denoted $D_F$, is an ideal of $\Ocal_S$ defined as follows
\begin{equation}
D_F=\underset{\Ical(F)}{\prod}\left(\dfrac{\det(\ell_{i_0},\cdots,\ell_{i_n})}{(\ell_{i_0})\cdots(\ell_{i_n})}\right)^2
\end{equation}
where $(\ell_i)$ denotes the ideal generated by the coefficients of $\ell_i$ and $\Ical(F)$ is the collection of $L$-linearly independent subsets $\lbrace \ell_{i_0},\ldots,\ell_{i_n}\rbrace$ of $\lbrace \ell_0,\ldots,\ell_t\rbrace$.
\end{dfn}

\begin{rem}
For each linear form $\ell_i$, the ideal $(\ell_i)$ is actually an ideal of the integral closure of $\Ocal_S$ in the splitting field $L$, but $D_F$ is an ideal of $\Ocal_S$. This follows from $D_F$ being invariant under $\Gal(L/K)$. For more details see the introduction of \cite{EvertseGyory}.
\end{rem}

\begin{rem}
If $F=\lambda\ell^{k_1}_1\cdots\ell^{k_t}_t$ is another representation of $F$, then we see that the scalar $\lambda$ is not used in the definition of the discriminant of $F$.  In particular, if $\gamma\in K^\times$, then $D_F=D_{\gamma F}$.
\end{rem}

\begin{dfn}
Let $F$ and $G$ be two decomposable forms in $n+1$ variables of degree $d$.  The forms $F$ and $G$ are \emph{weakly $\Ocal_S$-equivalent} if 
\begin{equation*}
F(X_0,\ldots ,X_n)=\lambda G(A(X_0,\ldots ,X_n))
\end{equation*}
for some $\lambda\in K^\times$ and some $A\in\GL_{n+1}(\Ocal_S)$.
\end{dfn}

To any subset $\Vcal\subset\PP^n(\Kbar)$ which is $\Gal(\Kbar/ K)$-stable we can associate a decomposable homogeneous form $F_\Vcal\in K[X_0,\ldots ,X_n]$ in the following manner:  Let $L$ denote the minimal splitting field of the set $\Vcal$.  For each point $P\in\Vcal$ let $\ell_P$ be the associated homogeneous form of degree $1$ for a choice of $L$-rational coordinates.  The form $\ell_P$ is well defined up to multiplication by a non-zero scalar of $L$.  Define

\begin{equation}
F_\Vcal=\prod_{P\in\Vcal}\ell_P
\end{equation}

As $\Vcal$ is $\Gal(\Kbar/K)$-stable, it follows that $F_\Vcal$ is a decomposable form of $K[X_0,\ldots ,X_n]$ of degree $|\Vcal|$ and is well defined up to multiplication by a non-zero scalar of $K$. Furthermore, it follows that $D_{F_\Vcal}$ is a well defined fractional ideal of $\Ocal_S$.  By the above remark, the discriminant ideal is unchanged if we multiply $F_\Vcal$ by a scalar of $\gamma\in K^\times$, so we may choose some $\gamma$ and, after multiplying $F_\Vcal$ through, assume that $F_\Vcal\in\Ocal_S[X_0,\ldots X_n]$.

\begin{dfn}\label{SIntSet}
Let $N\geq n+1$ be an integer and consider the set $\Pcal(S,N)$ of all subsets $\Vcal\subset\PP^n(\bar{K})$ such that the following conditions hold:
\begin{enumerate}
\item $\vert\Vcal\vert =N$.
\item $\Vcal$ is $\Gal(\Kbar/K)$-stable.
\item $\Vcal$ contains at least one linearly independent $(n+1)$-point subset.
\item $D_{F_\Vcal}=\Ocal_S$
\end{enumerate}
\end{dfn}

The condition that $D_{F_\Vcal}=\Ocal_S$ generalizes the familiar notion of pair-wise $S$-integrality for points of projective space. Recall that pair-wise $S$-integrality in $\PP^1$ merely requires that two distinct points $P,Q\in\PP^1(K)$ reduce to distinct points in $\PP^1(k_v)$ for every $v\not\in S$. The above requirement on the discriminant ideal is strictly stronger than pair-wise $S$-integrality as it requires that linear independent points of $\PP^n(\Kbar)$ not only remain distinct after reduction at each place $v$, but also that they remain linearly independent. In general, any place $v$ of $K$ for which the the valuation $v(D_{F_\Vcal})>0$ (of which there can only be finitely many) we have that some set of $n+1$ linearly independent points $P_0,\ldots ,P_n\in\Vcal$ descend to linearly dependent points in $\tilde{P_0},\ldots ,\tilde{P_n}\in\PP^n(k_v)$.

\begin{lem}\label{DeltaInv}
Let $\Vcal,\Wcal\subset\PP^n(\Kbar)$ be two $\Gal(\Kbar/K)$-invariant subsets such that $|\Vcal|=|\Wcal|=N$, both contain at least one linearly independent $(n+1)$-point subset, and $f(\Vcal)=\Wcal$ for some $f\in\PGL_{n+1}(\overline{\Ocal}_S)$. Then $D_{F_\Vcal}=D_{F_\Wcal}$. In particular, $D_{F_\Vcal}=\Ocal_S$ if and only if $D_{F_{f(\Vcal)}}=\Ocal_S$.
\end{lem}

\begin{proof}
Let $L/K$ be a finite extension such that $\Vcal,\Wcal$ are subsets of $\PP^n(L)$ and such that $f\in\PGL_n(\Ocal_T)$, where $\Ocal_T$ refers to the integral closure of $\Ocal_S$ in $L$. Let $A\in\GL_{n+1}(\Ocal_T)$ be a lift of $f$. It follows that $A^t$ defines a weak $\Ocal_S$-equivalence between $F_\Vcal$ and $F_\Wcal$. By Section 1 of \cite{EvertseGyory} the discriminant is invariant under weak $\Ocal_S$-equivalence. It follows that $D_{F_\Vcal}=D_{F_\Wcal}$.
\end{proof}

\begin{lem}
There exists a group action 
\begin{equation*}
\PGL_{n+1}(\Ocal_S)\times\Pcal(S,N)\rightarrow\Pcal(S,N)
\end{equation*}
defined by $(f,\Vcal)\mapsto f(\Vcal)=\lbrace f(P)|P\in\Vcal\rbrace$.
\end{lem}

\begin{proof}
Let $f,g,h\in\PGL_{n+1}(\Ocal_S)$ and $\Vcal\in\Pcal(S,N)$. 

We must show that conditions $(1)$ to $(4)$ in the definition of $\Pcal(S,N)$ hold for $f(\Vcal)$.  Requirements $(1)$ and $(4)$ are obviously satisfied since $f$ is an automorphism of $\PP^n$. Let $\sigma\in\Gal(\bar{K}/K)$ and $f(P)\in f(\Vcal)$.  Then, since $f$ is defined over $K=\mathrm{Frac}(\Ocal_S)$ and $\Vcal$ is $\Gal(\bar{K}/K)$-stable we have $\sigma\cdot f(P)=f(\sigma\cdot P)=f(Q)\in f(\Vcal)$ for some $Q\in\Vcal$, proving condition $(2)$. Lastly, $(4)$ follows from the preceding lemma.
\end{proof}

\begin{thm}\label{FiniteOrbit}
The group action of $\PGL_{n+1}(\Ocal_S)$ on $\Pcal(S,N)$ has only finitely many orbits.
\end{thm}

\begin{proof}
This theorem is a reformulation of Corollary 2 of Evertse and Gy\H{o}ry in \cite{EvertseGyory}. Their corollary states there are only finitely many weak $\Ocal_S$-equivalence classes of decomposable forms in $K[X_0,\ldots ,X_n]$ of fixed degree $N$ and given discriminant ideal $D$.

To conclude the proof of this theorem it suffices to show that if
\begin{equation*}
\Vcal=\lbrace P_1,\ldots ,P_N\rbrace,\Wcal=\lbrace Q_1,\ldots ,Q_N\rbrace\in\Pcal(S,N)
\end{equation*}
and if the forms $F_{\Vcal}$ and $F_{\Wcal}$ are weakly $\Ocal_S$-equivalent, then the sets $\Vcal$ and $\Wcal$ are in the same $\PGL_{n+1}(\Ocal_S)$-orbit.

Let $F$ and $G$ denote $F_{\Vcal}$ and $F_{\Wcal}$, respectively.  If $F$ and $G$ are weakly $\Ocal_S$-equivalent then there exists $\lambda\in K^\times$ and $A\in\GL_{n+1}(\Ocal_S)$ such that
\begin{equation*}
F(X_0,\ldots ,X_n)=\lambda G(A(X_0,\ldots ,X_n))
\end{equation*}
It follows that
\begin{equation*}
\underset{1\leq i\leq N}{\prod}\ell_{P_i}(X_0,\ldots ,X_n)=\lambda\underset{1\leq i\leq N}{\prod}\ell_{Q_i}(A((X_0,\ldots ,X_n))
\end{equation*}
After reordering we have that 
\begin{equation*}
\ell_{P_i}(X_0,\ldots ,X_n)=\lambda_i\ell_{Q_i}(A(X_0,\ldots ,X_n))
\end{equation*}
for $\lambda_i\in K^\times$. Equating coefficients gives that
\begin{equation*}
(p_{0i},\ldots ,p_{ni})=\lambda_i A^t(q_{0i},\ldots ,q_{ni}) 
\end{equation*}
where $A^t\in\GL_{n+1}(\Ocal_S)$ is the transpose of $A$. Let $a\in\PGL_{n+1}(\Ocal_S)$ be the corresponding projective linear transformation to $A^t$. Note that when we pass to projective space the scalars $\lambda_i$ become irrelevant.  Then $P_i=a(Q_i)$ and therefore $\Vcal=a(\Wcal)$. 
\end{proof}


\section{Main Theorem}
Let $M\geq 1$ and $\PrePer(\phi,M)$ denote the set of all points in $\PP^n(\Kbar)$ which are $\phi$-preperiodic of forward orbit has size at most $M$. It is known from the theory of canonical heights that the set $\PrePer(\phi,M)$ is finite (see \cite{silverman:ads}). Let $N=|\PrePer(\phi,M)|$. Every rational morphism of degree at least 2 has infinitely many preperiodic points, so by increasing $M$ we may assume that $N\geq n+2$ and moreover, by Fakhruddin's result on the Zariski density of preperiodic points (see Theorem 5.1 of \cite{Fak}) we may assume that there is a subset of $\PrePer(\phi,M)$ consisting of $n+2$ points which lie in general position.  

\begin{dfn}
A subset $\Vcal\subset\PP^n(\Kbar)$ with $|\Vcal|\geq n+1$ is in \emph{general position} if no $(n+1)$-point subset of $\Vcal$ lies in a hyperplane.
\end{dfn}

\begin{lem}\label{FixGenPos}
Let $\Vcal ,\Wcal\subset\PP^n(\Kbar)$ be finite , and assume that $\Vcal$ has a subset $\Vcal_0$ in general position with $|\Vcal_0|=n+2$. Then there exist only finitely many automorphisms $f\in\PGL_{n+1}(\Kbar)$ such that $f(\Vcal)=\Wcal$.
\end{lem}

\begin{proof}
Suppose the contrary and that $f_1,f_2,f_3,\ldots\in\PGL_{n+1}(\Kbar)$ is an infinite sequence of distinct automorphisms such that $f_i(\Vcal)=\Wcal$. Since $\Wcal$ is finite, it has only finitely many subsets, and we may assume, after perhaps passing to an infinite subsequence, that there exists a subset $\Wcal_0\subset\Wcal$ in general position with $|\Wcal_0|=n+2$ and $f_i(\Vcal_0)=\Wcal_0$ for all $i$. Choose $g\in\PGL_{n+1}(\Kbar)$ such that $g(\Wcal_0)=\Vcal_0$. Then the compositions $g\circ f_i$ form an infinite sequence of distinct automorphisms in $\PGL_{n+1}(\Kbar)$ such that $g\circ f_i(\Vcal_0)=\Vcal_0$. This gives a contradiction as there are only $(n+2)!$ such automorphisms.
\end{proof}

\begin{dfn}
A \emph{homogeneous lift} of $\phi$ is a map $\Phi:\AA^{n+1}\rightarrow\AA^{n+1}$ given by a $(n+1)$-tuple $(F_0,\ldots ,F_n)$ of forms $F_i$ such that $\phi(p_0:\ldots :p_n)=(F_0(p_0:\ldots:p_n):\ldots :F_n(p_0:\ldots :p_n))$ for all points $P=(p_0:\ldots :p_n)\in\PP^n$. An $\Ocal_S$-\emph{model} for a rational morphism $\phi:\PP^n\rightarrow\PP^n$ defined over $K$ is any conjugate of a homogeneous lift defined over $\Ocal_S$.
\end{dfn}

\begin{prop}\label{PrePerinPcal}
Assume that $\Ocal_S$ is a PID. Let $\phi,\psi:\PP^n\rightarrow\PP^n$ be rational morphisms defined over $K$ of degree $d$, both having good reduction at all places $v$ of $K$ outside $S$.  Assume that $[\psi]_K\in\Tw(\phi/K)$.
\begin{enumerate}[(a)]
\item There exist rational morphisms $\phi_0\in [\phi]_K$ and $\psi_0\in [\psi]_K$, and homogeneous lifts $\Phi,\Psi:\AA^{n+1}\rightarrow\AA^{n+1}$ of $\phi_0$ and $\psi_0$, respectively, such that $\Phi,\Psi$ have coefficients in $\Ocal_S$ and resultants $\Res(\Phi),\Res(\Psi)\in\Ocal^\times_S$.
\item There exists $A\in GL_{n+1}(\overline{\Ocal}_S)$ such that $\Psi=\Phi^A$.
\item For each integer $M\geq 1$, we have 
\begin{equation*}
\PrePer(\psi_0,M)=f(\PrePer(\phi_0,M)),
\end{equation*}
where $f:\PP^n\rightarrow\PP^n$ is the automorphism associated to $A$.  Moreover, $\PrePer(\phi_0,M)\in\Pcal(S,N)$ if and only if $\PrePer(\psi_0,M)\in\Pcal(S,N)$.
\end{enumerate}
\end{prop}

\begin{proof}
Part (a) follows from the main theorem of \cite{PetscheStout2} regarding the existence of global minimal models over principal ideal domains. The existence of $A\in\GL_{n+1}(\Kbar)$ follows immediately from $\Phi,\Psi$ being lifts of twists.  That $A\in\GL_{n+1}(\overline{\Ocal}_S)$ follows from Lemma 6 in \cite{PetscheSzpiroTepper}. For (c), that $\PrePer(\phi_0,M)=f(\PrePer(\psi_0,M))$ is immediate.  As $A\in\GL_{n+1}(\overline{\Ocal}_S)$ it follows that $f\in\PGL_{n+1}(\overline{\Ocal}_S)$, and it immediately follows from Lemma \ref{DeltaInv} that $\PrePer(\phi_0,M)\in\Pcal(S,N)$ if and only if the set $\PrePer(\psi_0,M)\in\Pcal(S,N)$.
\end{proof}

We are now ready to prove the main theorem.  

\begin{thm}\label{MainTheorem}
Let $\phi:\PP^n\rightarrow\PP^n$ be a rational morphism of degree $d>1$ defined over $K$ and $\Tw(\phi/K)$ the set of $K$-twists.  Then there are only finitely many twists $[\psi]_K\in\Tw(\phi/K)$ which have good reduction at all places $v\not\in S$.
\end{thm}

\begin{proof}
If no twists of $\phi$ have good reduction outside $S$ then there is nothing to prove.  Therefore, assume that at least one such twist exists, and since being twists is an equivalence relation, without loss of generality assume that $\phi$ has good reduction outside $S$. 

Assume contrary to the theorem that
\begin{equation}\label{TwistSeq}
[\psi_1]_K,[\psi_2]_K,[\psi_3]_K,\ldots
\end{equation}
is an infinite sequence of distinct twists which have good reduction outside S. 

We may assume that by Fakrhuddin's result on Zariski density of preperiodic points (see \cite{Fak}) that we can increase $M$ so that the set $\PrePer(\phi,M)$ contains a set of $n+2$ points of $\PP^n$ in general position.  Because $\phi$ is defined over $K$ the set $\PrePer(\phi,M)$ is $\Gal(\Kbar/K)$-stable.  There are only finitely many combinations of $n+1$ points of $\PrePer(\phi, M)$ which are linearly independent and let $D_0,\ldots D_t$ be their determinants.  There are only finitely many places $v$ of $K$ such that $v(D_i)>0$ for any $i=0,\ldots ,t$.  Enlarge $S$ by these places.  It follows that $D=\Ocal_S$ where $D$ is the discriminant associated to the set $\PrePer(\phi,M)$ and further enlarging $S$ does not change this condition. It follows that 
\begin{equation*}
\PrePer(\phi,M)\in\Pcal(S,N)
\end{equation*}
where $N=|\PrePer(\phi,M)|$. Finally, we may further enlarge $S$ until $\Ocal_S$ is a PID.  Consequently, by Proposition \ref{PrePerinPcal} we have that $\PrePer(\psi_i,M)\in\Pcal(S,N)$ for each $i=1,2,\ldots$, perhaps after replacing $\psi_i$ with some $K$-isomorphic morphism within $[\psi_i]_K$.

By Theorem \ref{FiniteOrbit}, we may assume, after passing to a subsequence, that $\PrePer(\phi, M), \PrePer(\psi_i,M)$ lie in the same $\PGL_{n+1}(\Ocal_S)$-equivalence class for all $i\geq 1$. It follows that there exists a sequence of linear transformations
\begin{equation*}
g_i\in\PGL_{n+1}(\Ocal_S)
\end{equation*}
such that 
\begin{equation*}
\PrePer(\psi_i,M)=g_i(\PrePer(\phi,M))
\end{equation*}
and therefore that
\begin{equation*}
\PrePer(\psi_i^{g_i},M)=g_i^{-1}(\PrePer(\psi,M))=\PrePer(\phi,M)
\end{equation*}
As $\psi_i^{g_i}$ also has good reduction at all $v\not\in S$, it suffices to replace $\psi_i$ with $\psi^{g_i}_i$ and assume that $\PrePer(\psi_i,M)=\PrePer(\phi,M)$ for all $i=1,2,\ldots$.

Let $f_i\in\PGL_{n+1}(\Kbar)$ be such that $\psi_i=\phi^{f_i}$. As the rational morphisms $\psi_i$ are assumed to be distinct, so must the $f_i$ be distinct, and it follows that $f_i$ gives a bijection
\begin{equation}\label{FiniteSetMap}
f_i:\PrePer(\phi,M)\rightarrow \PrePer(\phi,M)
\end{equation}
As $\PrePer(\phi,M)$ is finite and contains a subset of $n+2$ points in general position, it follows from Lemma \ref{FixGenPos} that only finitely many $f_i$ can exist and therefore gives the necessary contradiction.  
\end{proof}

\medskip

\bibliographystyle{acm}

\medskip

\end{document}